\pdfoutput=1
\RequirePackage{ifpdf}
\ifpdf % We are running pdfTeX in pdf mode
\documentclass[pdftex]{sigma}
\else
\documentclass{sigma}
\fi

\numberwithin{equation}{section}

\newtheorem{Theorem}{Theorem}[section]

\newtheorem{Lemma}[Theorem]{Lemma}
\newtheorem{Proposition}[Theorem]{Proposition}
 { \theoremstyle{definition}

\newtheorem{Note}[Theorem]{Note}
\newtheorem{Example}[Theorem]{Example}
\newtheorem{Remark}[Theorem]{Remark} }

\begin{document}

%\allowdisplaybreaks

\newcommand{\arXivNumber}{1610.09445}

\renewcommand{\thefootnote}{}

\renewcommand{\PaperNumber}{023}

\FirstPageHeading

\ShortArticleName{On Toric Poisson Structures of Type $(1,1)$ and their Cohomology}

\ArticleName{On Toric Poisson Structures of Type $\boldsymbol{(1,1)}$\\ and their Cohomology\footnote{This paper is a~contribution to the Special Issue ``Gone Fishing''. The full collection is available at \href{http://www.emis.de/journals/SIGMA/gone-fishing2016.html}{http://www.emis.de/journals/SIGMA/gone-fishing2016.html}}}

\Author{Arlo CAINE and Berit Nilsen GIVENS}
\AuthorNameForHeading{A.~Caine and B.N.~Givens}
\Address{California State Polytechnic University Pomona,\\ 3801 W.~Temple Ave., Pomona, CA, 91768, USA}
\Email{\href{mailto:jacaine@cpp.edu}{jacaine@cpp.edu}, \href{mailto:bngivens@cpp.edu}{bngivens@cpp.edu}}

\ArticleDates{Received October 29, 2016, in f\/inal form March 28, 2017; Published online April 06, 2017}

\Abstract{We classify real Poisson structures on complex toric manifolds of type $(1,1)$ and initiate an investigation of their Poisson cohomology. For smooth toric varieties, such structures are necessarily algebraic and are homogeneous quadratic in each of the distinguished holomorphic coordinate charts determined by the open cones of the associated simplicial fan. As an approximation to the smooth cohomology problem in each ${\mathbb C}^n$ chart, we consider the Poisson dif\/ferential on the complex of polynomial multi-vector f\/ields. For the algebraic problem, we compute $H^0$ and $H^1$ under the assumption that the Poisson structure is generically non-degenerate. The paper concludes with numerical investigations of the higher degree cohomology groups of $({\mathbb C}^2,\pi_B)$ for various~$B$.}


\Classification{53D17; 37J15}

\renewcommand{\thefootnote}{\arabic{footnote}}
\setcounter{footnote}{0}

\section{Introduction}

An almost complex structure $J$ on a smooth manifold $M$ determines a splitting of the complexif\/ied tangent bundle $TM\otimes_{\mathbb R} {\mathbb C}=T^{1,0}\oplus T^{0,1}$ into sub-bundles on which $J$ acts by $i$ and $-i$, respectively. This induces a decomposition
\begin{gather*}
\textstyle{\bigwedge^2} TM\otimes_{\mathbb R}{\mathbb C}=\big(T^{1,0}\wedge T^{1,0}\big)\oplus \big(T^{1,0} \wedge T^{0,1}\big)\oplus \big(T^{0,1}\wedge T^{0,1}\big)
\end{gather*}
and the smooth sections of $T^{1,0} \wedge T^{0,1}$ are termed bi-vector f\/ields on~$M$ of type~$(1,1)$. Observe that the canonical complex conjugation in the f\/ibers $TM\otimes_{\mathbb R} {\mathbb C}$ interchanges~$T^{1,0}$ and~$T^{0,1}$. Thus, $T^{1,0} \wedge T^{0,1}$ is stable under the complex conjugation on $\bigwedge^2 TM\otimes_{\mathbb R}{\mathbb C}$ determined by $X\wedge Y\mapsto \overline{X}\wedge \overline{Y}$ for each~$X$, $Y$ in a~f\/iber of $TM\otimes_{\mathbb R} {\mathbb C}$. However conjugation maps $T^{1,0}\wedge T^{1,0}$ to $T^{0,1}\wedge T^{0,1}$ and vice versa. A bi-vector f\/ield on $M$ is real if and only if it is f\/ixed by conjugation. So every real bi-vector f\/ield on $M$ is a sum of a bi-vector f\/ield of type $(1,1)$ together with a section of $T^{1,0}\wedge T^{1,0}$ plus its conjugate. In the context of this paper, we will be concerned with complex manifolds but focus on real bi-vector f\/ields of type $(1,1)$, i.e., smooth sections $\pi$ of $T^{1,0} \wedge T^{0,1}$ such that $\overline{\pi}=\pi$.

Suppose that $M$ is a complex toric manifold, i.e., a smooth complex manifold on which a~complex torus $T_{\mathbb C}$ acts holomorphically and ef\/fectively with an open dense orbit $\mathcal O$. We write~$T_{\mathbb C}$ for the torus, regarding it as the complexif\/ication of a real compact torus~$T$ in analogy with~${\mathbb T}_{\mathbb C}$ representing the non-zero complex numbers which is the complexif\/ication of the group~${\mathbb T}$ of complex numbers of modulus~$1$. By a toric Poisson structure on~$M$ we mean a Poisson structure which is invariant under the action of~$T_{\mathbb C}$. Such structures could be considered in the smooth or holomorphic category. For example, ${\mathbb C}^2$ with complex coordinates $(z,w)$ is a complex toric manifold acted on by ${\mathbb T}_{\mathbb C}^2$ via independent dilation of each complex coordinate and one can check that $\pi=-2izw\partial_z\wedge \partial_w$ is a holomorphic toric Poisson structure. Holomorphic Poisson structures are necessarily bi-vector f\/ields of type $(2,0)$. No bi-vector f\/ield of type $(1,1)$ can be holomorphic, nor can such a f\/ield occur as the real projection of holomorphic bi-vector f\/ield.

In this paper we will focus on the smooth category and consider structures of type $(1,1)$ which are generically non-degenerate. For example, ${\mathbb C}$ with coordinates $(z,\bar{z})$ is a complex toric manifold acted on by ${\mathbb T}_{\mathbb C}$ via $\zeta.(z,\bar{z})=(\zeta z,\overline{\zeta z})$ and $\pi=-2i z\bar{z}\partial_z\wedge \partial_{\bar{z}}$ is a smooth toric Poisson structure. Note that $\pi=-2i z\bar{z}\partial_z\wedge \partial_{\bar{z}}$ is real (f\/ixed by conjugation) and of type $(1,1)$. In terms of real variables, $z=x+iy$ and $\partial_z=\frac{1}{2}(\partial_x-i\partial_y)$ while $\partial_{\bar{z}}=\frac{1}{2}(\partial_x+i\partial_y)$ and one can check that
\begin{gather}\label{basic_example}
\pi=-2i z\bar{z}\partial_z\wedge \partial_{\bar{z}}=\big(x^2+y^2\big)\partial_x\wedge\partial_y,
\end{gather}
which means that $\pi$ is also a quadratic Poisson structure of elliptic type on ${\mathbb R}^2$ in the terminology of \cite{LiuXu}.

In \cite{Caine}, examples of toric Poisson structures of type $(1,1)$ on smooth toric varieties generalizing this example were constructed via a quotient construction. A smooth toric variety $M$ determines and is determined by a simplicial fan $\Sigma$ in ${\mathfrak t}$, the real Lie algebra of the torus $T$, which is simplicial with respect to the coweight lattice $\Lambda\subset{\mathfrak t}$ of $T_{\mathbb C}$. The variety can be recovered by a quotient ${\mathbb C}^d//N_{\mathbb C}\simeq M$ where the dimension $d$ and the sub-torus $N_{\mathbb C}$ of ${\mathbb T}_{\mathbb C}^d$ are determined from $\Sigma$. The double slash indicates that one does not divide ${\mathbb C}^d$ by the action of $N_{\mathbb C}$ but rather a dense open ${\mathbb T}_{\mathbb C}^d$-stable subset $\mathcal U_\Sigma$, determined from~$\Sigma$, on which $N_{\mathbb C}$ acts freely. The product Poisson structure $\pi\oplus \cdots \oplus \pi$ on ${\mathbb C}^d$, using $\pi$ from~\eqref{basic_example}, restricts to a real toric Poisson structure on~$\mathcal U_\Sigma$ of type~$(1,1)$ which is generically non-degenerate and, via the quotient map $\mathcal U_\Sigma\to {\mathbb C}^d//N_{\mathbb C}$, coinduces a real toric Poisson structure of type $(1,1)$ on $M$ which is generically non-degenerate. Each open cone in $\Sigma$ determines a distinguished holomorphic coordinate chart ${\mathbb C}^n$ on $M$ in which the action of $T_{\mathbb C}\simeq {\mathbb T}_{\mathbb C}^d/N_{\mathbb C}$ is equivalent to the action of ${\mathbb T}_{\mathbb C}^n$ on ${\mathbb C}^n$. In terms of such coordinates $(z_1,\dots,z_n)$, the Poisson structure on~$M$ has the form
\begin{gather}\label{model_for_toric_PS}
-2i\sum_{p,q=1}^n B_{pq}z_p\bar{z}_q\partial_{z_p}\wedge\partial_{\bar{z_q}}
\end{gather}
for some symmetric positive def\/inite integral matrix $[B_{pq}]$ determined by the open cone in $\Sigma$.

Let $Z_1=z_1\partial_{z_1},\dots,Z_n=z_n\partial_{z_n}$. Then we can write \eqref{model_for_toric_PS} in the form
\begin{gather*}
-2i\sum_{p,q=1}^n B_{pq} Z_p\wedge\overline{Z_q}
\end{gather*}
making it clear that it is the image of an element of ${\mathfrak t}_{\mathbb C}\wedge \overline{{\mathfrak t}_{\mathbb C}}$ under the natural map induced inf\/initesimally by the holomorphic action of $T_{\mathbb C}$ on $M$. In fact, since ${\mathfrak t}_{\mathbb C}$ is abelian and the map from ${\mathfrak t}_{\mathbb C}$ into sections of $T^{1,0}$ is a homomorphism of Lie algebras, it is clear that any element of ${\mathfrak t}_{\mathbb C}\wedge \overline{{\mathfrak t}_{\mathbb C}}$ maps to a toric Poisson structure on $M$ of type $(1,1)$. In terms of the coordinates $(z_1,\dots,z_n)$, such a structure has an expression of the form \eqref{model_for_toric_PS} for some matrix of complex numbers $[B_{pq}]$. In order for such a Poisson structure to be real, it turns out that $[B_{pq}]$ must be Hermitian and for it to be generically non-degenerate, it must be invertible. Given a formula such as \eqref{model_for_toric_PS} on ${\mathbb C}^d$, one could repeat the construction and coinduce a dif\/ferent real Poisson structure on a given toric variety $M$. From this slightly more general perspective, one can think of the construction in \cite{Caine} as simply that which begins with the toric Poisson structure on ${\mathbb C}^d$ corresponding to the identity matrix.

The purpose of this paper is to address the following natural questions regarding toric Poisson structures of type $(1,1)$.
\begin{enumerate}\itemsep=0pt
\item To what extent does this slight modif\/ication of the construction in \cite{Caine} classify all toric Poisson structures of type $(1,1)$?
\item What is the signif\/icance, if any, of the rationality of the coef\/f\/icients $[B_{pq}]$ of a toric Poisson structure as in \eqref{model_for_toric_PS}?
\item To what extent does the cohomology of a toric Poisson structure of type $(1,1)$ depend on the coef\/f\/icient matrix $[B_{pq}]$ in each chart and the topology of the toric variety $M$?
\end{enumerate}
In the f\/irst section of the paper we classify real toric Poisson structures of type $(1,1)$ establishing the following theorem.

\begin{Theorem}\label{1st_main_theorem}
Suppose that $(M,T_{\mathbb C})$ is a complex toric manifold. Then the real toric Poisson structures of type $(1,1)$ on $M$ are in one-to-one correspondence with the Hermitian forms on~${\mathfrak t}_{\mathbb C}^*$. We denote this correspondence $\pi_B\leftrightarrow B$. Furthermore, if $M$ is a smooth toric variety with corresponding simplicial fan $\Sigma$ in ${\mathfrak t}$, then in each distinguished holomorphic coordinate chart on~$M$ associated to an open cone of the fan for~$M$, $\pi_B$ has the form \eqref{model_for_toric_PS} where $[B_{pq}]$ is the matrix representing $B$ with respect to the integral basis spanning the cone.
\end{Theorem}

The formulas for $\pi_B$ in each of the distinguished holomorphic coordinate charts are then given by \eqref{model_for_toric_PS} for dif\/ferent matrices $[B_{pq}]$ which are congruent by elements of $GL(n,{\mathbb Z})$. Such structures are generically non-degenerate if and only if the Hermitian form $B$ is non-degenerate. We show in the following section that those forms which are rational with respect to the weight lattice $\Lambda^*$ in ${\mathfrak t}_{\mathbb C}^*$ then have an additional symmetry.

\begin{Theorem}\label{2nd_main_theorem}
Suppose that $(M,T_{\mathbb C})$ is a complex toric manifold and that $\pi_B$ is a generically non-degenerate real toric Poisson structure of type $(1,1)$ on $M$ associated to a non-degenerate Hermitian form $B$ on ${\mathfrak t}_{\mathbb C}^*$. Then the action of $T_{\mathbb C}$ on the open orbit $\mathcal O$ is Hamiltonian with respect to $\pi_B$ if and only if $B$ is rational with respect to the weight lattice $\Lambda^*\subset{\mathfrak t}_{\mathbb C}^*$.
\end{Theorem}

In the remaining sections of the paper, we initiate an investigation of the cohomology of such structures, focusing on $M={\mathbb C}^2$ and various example forms~$B$ in an ef\/fort to begin addressing the much harder third question. These examples remove topology from consideration because~${\mathbb C}^2$ is contractible and focus only on contributions to the cohomology from the degeneracies of the Poisson tensor. Since toric Poisson structures are algebraic in the distinguished coordinate charts, we consider the cohomology problem f\/irst in the algebraic category as an approximation to the smooth problem. The formal cohomology, in this sense, of quadratic Poisson structures was computed generally in~\cite{Monnier} although it takes some work to interpret and apply the results in this setting. We describe how to use symbolic computation with Mathematica to compute the cohomology for multivector f\/ields of low monomial degree. These numerical results conf\/irm the results of~\cite{Monnier} in these cases and suggest that the cohomology of $({\mathbb C}^2,\pi_B)$ is quite small for generic $B$ and, in particular, is f\/inite dimensional. Rather surprisingly, we found an example of a $({\mathbb C}^2,\pi_B)$ where $B$ is symmetric positive def\/inite and integral, yet the cohomology is inf\/inite dimensional, having inf\/initely many generators in wedge degrees 3 and 4. This example is the local form of the toric Poisson structure constructed in \cite{Caine} on one of the Hirzebruch surfaces. In general, classes in $H^0(M,\pi)$, $H^1(M,\pi)$, and $H^2(M,\pi)$ have well-known interpretations, but the interpretation of non-trivial classes of higher wedge degree is unclear. It is our hope that our examples will be studied further so that interpretations of higher degree classes can be found.

If $\pi$ is a generically non-degenerate real toric Poisson structure of type $(1,1)$ on a complex toric manifold $X$, it is easy to see that $H^0(M,\pi)={\mathbb C}\langle 1\rangle$ if $X$ is connected. It is a more subtle fact that ${\mathfrak t}_{\mathbb C}$ includes in $H^1(M,\pi)$ as a real subspace if the fan for $M$ is closed in ${\mathfrak t}$. The main result of the later sections is that there is nothing more in $H^1({\mathbb C}^n,\pi_B)$ when $B$ is non-degenerate, at least in the algebraic category.

\begin{Theorem}\label{3rd_main_theorem}
If $[B_{pq}]$ is an invertible $n\times n$ Hermitian matrix, then for the algebraic Poisson cohomology, the inclusion of ${\mathfrak t}_{\mathbb C}$ into $H^1({\mathbb C}^n,\pi_B)$ is an isomorphism
\begin{gather*}
H^1({\mathbb C}^n,\pi_B)\simeq {\mathfrak t}_{\mathbb C}
\end{gather*}
of real vector spaces.
\end{Theorem}

In a forthcoming paper, we will use representation theory to completely determine the Poisson cohomology of $({\mathbb C}^n,\pi_B)$, reducing the determination of generators for the cohomology to the solution of certain Diophantine inequalities depending on the matrix $B$. This work will reinterpret the computations of \cite{Monnier}.

\subsection*{Conventions}

Throughout the paper ${\mathbb T}$ denotes the Lie group of complex numbers of modulus one. The Lie algebra of ${\mathbb T}$ will be identif\/ied with $i{\mathbb R}$ via the map $i{\mathbb R}\to{\mathbb T}$ given by $s\mapsto e^s$. The kernel of this group homomorphism is the lattice $\tau i{\mathbb Z}$ in $i{\mathbb R}$ where $\tau$ denotes the circumference of the unit circle. We use $\tau$ instead of $2\pi$ to avoid a clash of notation when this number and a Poisson structure labeled $\pi$ need to appear in the same formula. By taking products, this induces an identif\/ication of the Lie algebra of ${\mathbb T}^d$ with $(i{\mathbb R})^d$. The kernel of the product map $(i{\mathbb R})^d\to {\mathbb T}^d$ is the lattice $(\tau i{\mathbb Z})^d\subset(i{\mathbb R})^d$. We will write $e_1,e_2,\dots,e_d$ for the standard basis of~${\mathbb Z}^d$ so that $\tau i e_1,\dots,\tau ie_d$ generate $(\tau i{\mathbb Z})^d$. Writing ${\mathbb T}_{\mathbb C}$ for the complex Lie group of non-zero complex numbers and similarly identifying the Lie algebra of ${\mathbb T}_{\mathbb C}$ with ${\mathbb C}$ via the map ${\mathbb C}\to {\mathbb T}_{\mathbb C}$ given by $s\mapsto e^s$, we thus extend the identif\/ication of the real Lie subalgebra of~${\mathbb T}$ with the real subspace~$i{\mathbb R}$ of~${\mathbb C}$.

Let $T$ denote a compact connected abelian Lie group of real dimension $n$, i.e., a compact torus, and let $T_{\mathbb C}$ denote its complexif\/ication. The kernel of the group homomorphism ${\mathfrak t}_{\mathbb C}\to T_{\mathbb C}$ def\/ined by $\xi\mapsto \exp(\xi)$ is the coweight lattice $\Lambda\subset {\mathfrak t}\subset {\mathfrak t}_{\mathbb C}$. If we choose a basis $\xi_1,\dots,\xi_n$ for~$\Lambda$ and def\/ine a~${\mathbb Z}$-linear map $(\tau i{\mathbb Z})^n\to \Lambda$ by $\tau ie_k\mapsto \xi_k$, then by extending to a ${\mathbb C}$-linear map ${\mathbb C}^n\to {\mathfrak t}_{\mathbb C}$ we obtain a commutative diagram
\begin{gather*}
\begin{matrix}
0 & \longrightarrow & (\tau i{\mathbb Z})^n & \longrightarrow & {\mathbb C}^n & \longrightarrow & ({\mathbb T}_{\mathbb C})^n & \longrightarrow & 1 \\
 & & \big\downarrow & & \big\downarrow & & & & \\
0 & \longrightarrow & \Lambda & \longrightarrow & {\mathfrak t}_{\mathbb C} & \longrightarrow & T_{\mathbb C} & \longrightarrow & 1
\end{matrix}
\end{gather*}
with exact rows and in which the two vertical arrows are isomorphisms. This induces an isomorphism $({\mathbb T}_{\mathbb C})^n\to T_{\mathbb C}$ making the diagram commute. If $(z_1,\dots,z_n)=(e^{s_1},\dots,e^{s_n})$ in~$({\mathbb T}_{\mathbb C})^n$ then $(z_1,\dots,z_n)$ corresponds to
\begin{gather*}
\exp\left(\frac{s_1}{\tau i}\xi_1+\dots+\frac{s_n}{\tau i}\xi_n\right)=\exp\left(\frac{s_1}{\tau i}\xi_1\right)\cdots \exp\left(\frac{s_n}{\tau i}\xi_n\right)
\end{gather*}
in $T_{\mathbb C}$.

\begin{Note}\label{isom_with_std_torus}
For $z\not=0$ in ${\mathbb C}$ and $\xi\in \Lambda$, we write $z^\xi$ for $\exp(\frac{1}{\tau i}\log(z)\xi)\in T_{\mathbb C}$. This is well def\/ined because $\exp(k\xi)=1$ for each $k\in{\mathbb Z}$. Thus, for a choice of basis $\xi_1,\dots,\xi_n$ of $\Lambda$, we obtain an isomorphism $(T_{\mathbb C})^n\to T_{\mathbb C}$ given by $(z_1,\dots,z_n)\mapsto z_1^{\xi_1}\cdots z_n^{\xi_n}$.
\end{Note}

\section[Real toric Poisson structures of type $(1,1)$]{Real toric Poisson structures of type $\boldsymbol{(1,1)}$}

Let $(M,T_{\mathbb C})$ be a complex toric manifold and choose a ${\mathbb C}$-basis $\xi_1,\dots,\xi_n$ for ${\mathfrak t}_{\mathbb C}$. Then through the holomorphic action of $T_{\mathbb C}$ we obtain corresponding holomorphic $T_{\mathbb C}$-invariant vector f\/ields $X_1,\dots,X_n$ on $M$, i.e., holomorphic sections of $T^{1,0}$, which are $T_{\mathbb C}$-invariant. Given a Hermitian matrix of complex numbers, one can use this data to construct a real toric Poisson structure on~$M$ of type~$(1,1)$.

\begin{Lemma} \label{model_for_toric_PS_lemma}
If $B=[B_{pq}]$ is a Hermitian matrix of complex numbers, then
\begin{gather}\label{toric_representation}
\pi_B=2i\sum_{p,q=1}^n B_{pq}X_p\wedge \overline{X_q}
\end{gather}
is a real toric Poisson structure on $M$ of type $(1,1)$.
\end{Lemma}
\begin{proof}
By construction, $\pi_B$ is a $T_{\mathbb C}$-invariant bi-vector f\/ield on $M$ on type $(1,1)$. Since
\begin{gather*}
\overline{\pi_B} =  \overline{2i\sum_{p,q=1}^n B_{pq}X_p\wedge \overline{X_q}} = -2i\sum_{p,q=1}^n \overline{B_{pq}}\,\overline{X_p}\wedge X_q \\
\hphantom{\overline{\pi_B}}{} = 2i\sum_{p,q=1}^n \overline{B_{pq}}\, X_q\wedge \overline{X_p} = 2i\sum_{p,q=1}^n \overline{B_{qp}}\,X_p\wedge \overline{X_q} \end{gather*}
we see that $\pi_B$ is real ($\overline{\pi_B}=\pi_B$) because $\overline{B_{qp}}=B_{pq}$ for each $p,q=1,\dots,n$. The holomorphic action of $T_{\mathbb C}$ on $M$ induces a homomorphism of Lie algebras ${\mathfrak t}_{\mathbb C}\to T^{1,0}$ and this carries $\xi_p\to X_p$ for each $p$ by def\/inition. Thus, $[X_p,X_{p'}]=0$ for each $p,p'=1,2,\dots,n$. Since conjugation preserves brackets, we know that $[\overline{X_q},\overline{X_{q'}}]=0$ for all $q,q'=1,2,\dots,n$. Recalling that holomorphic derivations and anti-holomorphic derivations commute, we f\/ind that $[X_p,\overline{X_q}]=0$ for all $p,q=1,2,\dots,n$. Thus $\pi_B$ is a linear combination of wedge products of commuting vector f\/ields and hence $[\pi_B,\pi_B]=0$ by linearity and the graded Leibniz rule for the Schouten bracket.
\end{proof}

The next result shows that every real toric Poisson structure of type $(1,1)$ on $M$ is of the form $\pi_B$. The representation \eqref{toric_representation} depends on the basis $\xi_1,\dots,\xi_n$ which determines both the holomorphic vector f\/ields $X_1,\dots,X_n$ and, as we will see, the matrix $[B_{pq}]$.

\begin{Theorem}\label{invariant_nondegenerate_PS_on_tori}
Suppose that $\pi$ is a real toric Poisson structure on $M$ of type $(1,1)$. For each ${\mathbb C}$-basis $\xi_1,\dots,\xi_n$ of ${\mathfrak t}_{\mathbb C}$ there exists an $n\times n$ Hermitian matrix $[B_{pq}]$ such that $\pi=\pi_B$ as in~\eqref{toric_representation}.
\end{Theorem}
\begin{proof}
The key point is that a ${\mathbb C}$-basis $\xi_1,\dots,\xi_n$ for ${\mathfrak t}_{\mathbb C}$ determines a frame $X_1,\dots,X_n$ for $T^{1,0}$ over the open dense $T_{\mathbb C}$-orbit $\mathcal O$. This is because the action of $T_{\mathbb C}$ is free and transitive on $\mathcal O$. By conjugation, $\overline{X_1},\dots,\overline{X_n}$ is then a frame for $T^{0,1}$ over $\mathcal O$. Consequently, the bi-vector f\/ields $X_p\wedge\overline{X_q}$ for $p,q=1,2,\dots,n$ form a frame for $T^{1,0}\wedge T^{0,1}$ over $\mathcal O$. Thus, there exist smooth complex valued functions $B_{pq}$ on $\mathcal O$ such that
\begin{gather}\label{pi_rep_on_O}
\pi=2i\sum_{p,q=1}^n B_{pq}X_p\wedge \overline{X_q}
\end{gather}
on $\mathcal O$. As we saw in the proof of Lemma \ref{model_for_toric_PS_lemma}, \eqref{pi_rep_on_O} implies that
\begin{gather*}
\overline{\pi}=2i\sum_{p,q=1}^n \overline{B_{qp}}\,X_p\wedge \overline{X_q}
\end{gather*}
on $\mathcal O$. Since $\pi$ is real, we can match coef\/f\/icients to conclude that $B_{pq}=\overline{B_{qp}}$ for each $p,q=1,\dots,n$, so that $[B_{pq}]$ is a Hermitian matrix of functions. It remains to show that these functions are constant.

If $\pi$ is also $T_{\mathbb C}$-invariant, then the Schouten bracket $[X_r,\pi]=0$ for each $r=1,\dots,n$. Observe that
\begin{gather*}
[X_r,B_{pq}X_p\wedge \overline{X_q}]=X_r(B_{pq})X_p\wedge \overline{X_q}+B_{pq}[X_r,X_p]\wedge \overline{X_q} + B_{pq}X_p\wedge [X_r,\overline{X_q}] \\
\hphantom{[X_r,B_{pq}X_p\wedge \overline{X_q}]}{} =X_r(B_{pq})X_p\wedge \overline{X_q}
\end{gather*}
because the Schouten bracket is a graded derivation. Thus, $[X_r,\pi]=0$ on ${\mathbb T}_{\mathbb C}^n$ for each $r=1,\dots,n$ if and only if $X_r(B_{pq})=0$ on $\mathcal O$ for each $p,q,r=1,\dots,n$. Hence, each $B_{pq}$ is anti-holomorphic on $\mathcal O$. Likewise, the fact that $[\overline{X_r},\pi]=0$ for all $r=1,\dots,n$ can be used to argue that $B_{pq}$ is also holomorphic on $\mathcal O$. Therefore, each $B_{pq}$ is constant on $\mathcal O$. But since the vector f\/ields $X_p$ and $\overline{X_q}$ are smooth on $M$ and the functions $B_{pq}$ are constant on the open dense set~$\mathcal O$, the representation~\eqref{pi_rep_on_O} is actually valid on all of $M$. Hence $\pi=\pi_B$ on $M$ as desired.
\end{proof}

\begin{Proposition}\label{generically_nondeg_prop}
If $\pi=\pi_B$ is a real toric Poisson structure of type $(1,1)$ on $M$ which is generically non-degenerate, then $[B_{pq}]$ is invertible.
\end{Proposition}
\begin{proof}
Since $\pi_B$ is $T_{\mathbb C}$-invariant,
\begin{gather*}
\pi_B^n=(2i)^n\det([B_{pq}])X_1\wedge\dots \wedge X_n\wedge \overline{X_1}\wedge \dots \wedge \overline{X_n}
\end{gather*}
is $T_{\mathbb C}$-invariant. Since $X_1\wedge\dots \wedge X_n\wedge \overline{X_1}\wedge \dots \wedge \overline{X_n}$ is non-zero on $\mathcal O$ we conclude that $\det([B_{pq}])\not=0$ on $\mathcal O$ if $\pi$ is generically non-degenerate. Hence $[B_{pq}]$ is invertible.
\end{proof}

\begin{Remark}
By applying Theorem \ref{invariant_nondegenerate_PS_on_tori} for a dif\/ferent choice of basis $\xi_1',\dots,\xi_n'$ of ${\mathfrak t}_{\mathbb C}$, one obtains another Hermitian matrix $[B_{pq}']$ such that $\pi=\pi_{B'}$. One can check that the matri\-ces~$[B_{pq}]$ and~$[B_{pq}']$ are congruent via the change of basis matrix from $\xi_1,\dots,\xi_n$ to $\xi_1',\dots,\xi_n'$. Thus, these matrices represent the same Hermitian form on ${\mathfrak t}_{\mathbb C}^*$. The f\/irst half of Theorem~\ref{1st_main_theorem} from the introduction has now been established. The real toric Poisson structures on~$M$ of type~$(1,1)$ are in one-to-one correspondence with the Hermitian forms on~${\mathfrak t}_{\mathbb C}^*$.
\end{Remark}

If $M$ is a smooth compact toric variety with associated simplicial fan $\Sigma$ in ${\mathfrak t}$, then we obtain a~distinguished family of holomorphic coordinate charts on $M$, one for each open cone in $\Sigma$. Such a cone is spanned by a ${\mathbb Z}$-basis $\xi_1,\dots,\xi_n$ of the coweight lattice $\Lambda$ in ${\mathfrak t}$ and with this basis we obtain an isomorphism ${\mathbb T}_{\mathbb C}^n\to T_{\mathbb C}$ given by $(z_1,\dots,z_n)\mapsto z_1^{\xi_1}\dots z_n^{\xi_n}$ as in Note \ref{isom_with_std_torus}. By choosing a point $x\in\mathcal O$, we obtain a map ${\mathbb T}_{\mathbb C}^n\to T_{\mathbb C}.x=\mathcal O$ and holomorphic coordinates on $\mathcal O$. In terms of these coordinates, the holomorphic vector f\/ields $X_1,\dots,X_n$ associated to $\xi_1,\dots,\xi_n$ have the form $X_p=z_p\partial_{z_p}$. Thus
\begin{gather}\label{pi_B_in_coordinates}
\pi=i\sum_{p,q=1}^n B_{pq} z_p\bar{z}_q \partial_{z_p}\wedge\partial_{\bar{z}_q}
\end{gather}
in these coordinates.

In~\cite{Caine}, local forms such as these were obtained for Poisson structures on smooth toric varieties which were coinduced from a product Poisson structure on ${\mathbb C}^d$ via a quotient construction. There, however, the matrices $[B_{pq}]$ were also integral and thus associated to integral Hermitian forms on~${\mathfrak t}_{\mathbb C}^*$. We will next show that toric Poisson structures of type $(1,1)$ associated to non-degenerate Hermitian forms which are rational with respect to $\Lambda$ are precisely those for which the action of~$T_{\mathbb C}$ on~$\mathcal O$ is Hamiltonian.

The sense in which we will consider the action to be Hamiltonian is the following from \cite{Lu}. Regarding $T_{\mathbb C}$ as a Poisson Lie group with the trivial Poisson Lie group structure, we can think of the action $ (\mathcal O,\pi_B)\times (T_{\mathbb C},0) \to (\mathcal O,\pi_B)$ as a Poisson action and search for a momentum mapping for the action taking values in the dual Poisson Lie group. However, we must specify which model of the dual group we are considering. Since $T_{\mathbb C}$ is abelian, the simply connected model for the dual group is simply the vector space ${\mathfrak t}_{\mathbb C}^*$. But an alternative is the quotient ${\mathfrak t}_{\mathbb C}^*/\Lambda^*$ which is isomorphic to $T_{\mathbb C}$. With this model, we can regard $(T_{\mathbb C},0)$ as a self-dual Poisson Lie group. Our f\/inal goal in this section is to classify which $\pi_B$ admit such a momentum map with values in $T_{\mathbb C}$ or some f\/inite quotient of $T_{\mathbb C}$.

Suppose that $(N,\pi)\times (G,\pi_G)\to (N,\pi)$ is a Poisson action of a Poisson Lie group $(G,\pi_G)$ on a Poisson manifold $(N,\pi)$. Let $\mu\colon N\to G^*$ be a smooth map from $(N,\pi)$ to a dual group $G^*$ of $(G,\pi_G)$. Recall that $\mu$ is a $G^*$-valued \emph{momentum map}~\cite{Lu} if for each $\xi\in{\mathfrak g}$, the Lie algebra of~$G$,
\begin{gather}\label{differential condition}
X_\xi=\pi^\#(\mu^*(\theta_\xi)),
\end{gather}
where $X_\xi$ is the vector f\/ield on $N$ generated by the inf\/initesimal action of $\xi$, $\theta_\xi$ is the right-invariant 1-form on $G^*$ generated by $\xi\in {\mathfrak g}=(T_eG^*)^*$, and $\mu^*\colon T^*G^*\to T^*N$ is the cotangent lift of $\mu\colon N\to G^*$. We will say that a Poisson action of $(G,\pi_G)$ on $(N,\pi)$ is $G^*$-Hamiltonian if there exists a $G^*$-valued momentum map for the action.

\begin{Theorem}
Suppose $\pi_B$ is a real toric Poisson structure on $M$ of type $(1,1)$ associated to a~non-degenerate Hermitian form $B$ on ${\mathfrak t}_{\mathbb C}^*$. Then the action of $T_{\mathbb C}$ on $\mathcal O$ is $T_{\mathbb C}$-Hamiltonian if and only if $B$ is rational with respect to $\Lambda^*$ in ${\mathfrak t}_{\mathbb C}^*$.
\end{Theorem}
\begin{proof}
We will compute in the global coordinates $(z_1,\dots,z_n)$ for $\mathcal O$ associated to an integral basis $\xi_1,\dots,\xi_n$ for $\Lambda$. This identif\/ies $\mathcal O$ with ${\mathbb T}_{\mathbb C}^n$ and ensures that $\pi_B$ has the form \eqref{pi_B_in_coordinates} where $[B_{pq}]$ is the matrix representing $B$ in terms of that basis. Furthermore, we can regard $T_{\mathbb C}$ as ${\mathbb T}_{\mathbb C}^n$ acting on ${\mathbb T}_{\mathbb C}^n$ by multiplication. So, to specify a $G^*=T_{\mathbb C}$-valued momentum map on $\mathcal O$ is the same as specifying a map $\mu\colon {\mathbb T}_{\mathbb C}^n\to{\mathbb T}_{\mathbb C}^n$ satisfying the dif\/ferential condition~\eqref{differential condition}. To proceed, we will classify the local solutions $\mu$ to \eqref{differential condition} and show that these local solutions can be extended globally to ${\mathbb T}_{\mathbb C}^n$ if and only if $[B_{pq}]$ is a rational matrix.

If $\xi=x_1e_1+\dots+x_ne_n$ and $\eta=y_1e_1+\dots+y_ne_n$ in ${\mathbb C}^n$, let $\langle \xi,\eta\rangle = \sum\limits_{k=1}^n \bar{x}_ky_k$ so that $\langle\cdot,\cdot\rangle$ is the standard Hermitian inner product on ${\mathbb C}^n$ which is conjugate linear in its f\/irst argument. We will identify the real dual of ${\mathbb C}^n$ with ${\mathbb C}^n$ using $\mathrm{Im}\langle\cdot,\cdot\rangle$. Then $\xi\in{\mathbb C}^n$ with $\xi=\sum\limits_{k=1}^n x_ke_k$ generates a holomorphic vector f\/ield $X_\xi$ on $\mathcal O$ and the invariant 1-form
\begin{gather*}
\theta_\xi = \operatorname{Re}\left(\sum_{k=1}^n \overline{ix}_k\frac{dw_k}{w_k}\right)
\end{gather*}
on the dual group $G^*={\mathbb T}_{\mathbb C}^n$ with coordinates $(w_1,\dots,w_n)$. To verify that this formula for $\theta_\xi$ is correct, observe that $\theta_\xi$ is invariant under multiplication by $(\zeta_1,\dots,\zeta_n)\in {\mathbb T}_{\mathbb C}^n$ and restricts to the form $\operatorname{Re}\langle i\xi,\cdot\rangle=\mathrm{Im}\langle \xi,\cdot\rangle$ at the identity.

Let $[B^{pq}]$ denote the matrix inverse to $[B_{pq}]$, so that $[\overline{B^{pq}}]$ is the matrix representing the conjugate Hermitian form dual to $B$ on ${\mathfrak t}_{\mathbb C}^*$. We set
\begin{gather*}%\label{toric_momentum_map}
\mu(z_1,\dots,z_n)=\left(z_1^{\overline{B^{11}}}\cdots z_n^{\overline{B^{n1}}},\dots,z_1^{\overline{B^{1n}}}\cdots z_n^{\overline{B^{nn}}}\right)=(w_1,\dots,w_n)
\end{gather*}
and obtain a multi-valued function on $M={\mathbb T}_{\mathbb C}^n$ with values in ${\mathbb T}_{\mathbb C}^n$. We will show that $\mu$ is the unique solution to \eqref{differential condition} satisfying $\mu(1,\dots,1)=(1,\dots,1)$. By inspection, we see that $\mu$ can be interpreted as a single-valued function on $M={\mathbb T}_{\mathbb C}^n$ with values in a f\/inite quotient of $T_{\mathbb C}$ if and only if each $\overline{B^{pq}}$ is rational. But this of course requires that each $B_{pq}$ be rational as well.

Observe that
\begin{gather*}
\mu^*\left(\frac{dw_k}{w_k}\right)=\frac{d\big(z_1^{\overline{B^{1k}}}\cdots z_n^{\overline{B^{nk}}}\big)}{z_1^{\overline{B^{k1}}}\cdots z_n^{\overline{B^{kn}}}}=\sum_{j=1}^n \overline{B^{jk}}\frac{dz_j}{z_{j}}
\end{gather*}
for each $k=1,2,\dots,n$. Thus,
\begin{gather*}
\mu^*(\theta_\xi)=\operatorname{Re}\left(\sum_{k=1}^n \overline{ix_k}\left(\sum_{j=1}^n \overline{B^{jk}}\frac{dz_j}{z_{j}}\right)\right)=\operatorname{Re}\left(\sum_{j,k=1}^n \overline{ix_k}B^{kj}\frac{dz_j}{z_{j}}\right)
\end{gather*}
because $[B^{jk}]$ is Hermitian. Applying $\pi_B^\#$ from \eqref{pi_B_in_coordinates}, we obtain
\begin{gather}
\pi_B^\#(\mu^*(\theta_\xi)) =\pi_B^\#\left(\operatorname{Re}\left(\sum_{j,k=1}^n \overline{ix_k}B^{kj}\frac{dz_j}{z_{j}} \right)\right)
=\operatorname{Re}\left(\pi_B^\#\left(\sum_{j,k=1}^n \overline{ix_k}B^{kj}\frac{dz_j}{z_{j}}\right)\right) \label{step1}
\end{gather}
since $\pi_B^\#$ is real. Now,
\begin{gather}
\pi_B^\#\left(\sum_{j,k=1}^n \overline{ix_k}B^{kj}\frac{dz_j}{z_{j}}\right)=2\sum_{p,q=1}^n \sum_{j,k=1}^n B_{pq}B^{kj}\frac{\bar{x}_k z_p \bar{z}_q}{z_j} (\partial_{z_p}\wedge \partial_{\bar{z}_q})^\#(dz_j) \nonumber \\
\hphantom{\pi_B^\#\left(\sum_{j,k=1}^n \overline{ix_k}B^{kj}\frac{dz_j}{z_{j}}\right)}{}
=2\sum_{p,q=1}^n \sum_{j,k=1}^n B_{pq}B^{kj}\frac{\bar{x}_k z_p \bar{z}_q}{z_j} \delta_{pj}\partial_{\bar{z}_q} \nonumber \\
\hphantom{\pi_B^\#\left(\sum_{j,k=1}^n \overline{ix_k}B^{kj}\frac{dz_j}{z_{j}}\right)}{}
=2\sum_{p,q=1}^n \sum_{k=1}^n B_{pq}B^{k p}\bar{x}_k\bar{z}_q\partial_{\bar{z}_q} =2\sum_{k=1}^n\bar{x}_k\bar{z}_k\partial_{\bar{z}_k} \label{step2}
\end{gather}
since $[B^{kp}][B_{pq}]=[\delta_p^k]$. Combining \eqref{step1} and \eqref{step2}, we have
\begin{gather*}
\pi_B^\#(\mu^*(\theta_\xi)) =\operatorname{Re}\left(2\sum_{k=1}^n\bar{x}_k\bar{z}_k\partial_{\bar{z}_k}\right)=\sum_{k=1}^n x_k z_k\partial_{z_k} +\text{c.c.} = 2\operatorname{Re} (X_\xi)
\end{gather*}
as required.

Thus, $\mu$ is locally a solution to \eqref{differential condition}. All other solutions dif\/fer from $\mu$ by post-multiplication by an element $(\zeta_1,\dots,\zeta_n)$ of ${\mathbb T}_{\mathbb C}^n$ which ensures that $(1,\dots,1)\mapsto (\zeta_1,\dots,\zeta_n)$. This completes the proof.
\end{proof}

\begin{Example}
For $\pi_b=-2ibz\bar{z}\partial_z\wedge\partial_{\bar{z}}$ on ${\mathbb T}_{\mathbb C}$, $b\in{\mathbb R}\setminus\{0\}$, the map $\mu(z)=z^{\frac{1}{b}}$ satisf\/ies the dif\/ferential condition~\eqref{differential condition} with respect to the inf\/initesimal action of ${\mathbb T}_{\mathbb C}$ on itself. However, $\mu$~is single-valued only if $\frac{1}{b}\in{\mathbb Z}$. For general $b=\frac{m}{n}\in\mathbb Q$, this can be remedied by replacing the codomain ${\mathbb T}_{\mathbb C}$ of $\mu$ with a f\/inite quotient of~${\mathbb T}_{\mathbb C}$. Since each of those quotients are isomorphic to~${\mathbb T}_{\mathbb C}$, this does not substantially change the model for the dual group.
\end{Example}

\section{Algebraic Poisson cohomology}

In complex dimension $n=1$, there are three examples of complex toric manifolds, namely ${\mathbb T}_{\mathbb C}$, ${\mathbb C}={\mathbb T}_{\mathbb C}\cup\{0\}$, or ${\mathbb C} P^1={\mathbb C}\cup\{\infty\}$. Applying Theorem \ref{1st_main_theorem}, the real toric Poisson structures of type $(1,1)$ have the form
\begin{gather}\label{pi_B_n_is_1}
\pi_b=-2ib\,z\bar{z}\,\partial_{z}\wedge\partial_{\bar{z}}
\end{gather}
for some $b\in{\mathbb R}$ with $b\not=0$ in terms of the holomorphic coordinate on ${\mathbb T}_{\mathbb C}$. In terms of the underlying real variables $z=x+iy$, this Poisson structure has the expression
\begin{gather*}
\pi_b=b\big(x^2+y^2\big)\partial_x\wedge\partial_y.
\end{gather*}
In the example $M={\mathbb C}$, it is therefore a quadratic plane Poisson structure of elliptic type. On ${\mathbb R}^2\setminus\{(0,0)\}={\mathbb T}_{\mathbb C}$ it is non-degenerate and so its Poisson cohomology is isomorphic to the de Rham cohomology
of ${\mathbb R}^2\setminus\{(0,0)\}$. In \cite{Nakanishi}, Nakanishi performed a careful analysis of the partial dif\/ferential equations encoded in the condition $\sigma(Y)=0$, where $\sigma$ is the Poisson dif\/ferential for~$\pi_b$ and $Y$ is a smooth multivector f\/ield on ${\mathbb R}^2$, and through that analysis determined the Poisson cohomology of $\pi_b$ on ${\mathbb C}$. Since the expression~\eqref{pi_B_n_is_1} is invariant under the change of variables $z\mapsto 1/z$, a Mayer--Vietoris argument can then be applied to compute the Poisson cohomology of $\pi_b$ on ${\mathbb C} P^1$ using Nakanishi's result on~${\mathbb T}_{\mathbb C}\cup\{0\}={\mathbb C}$ and ${\mathbb T}_{\mathbb C}\cup\{\infty\}\simeq {\mathbb C}$. This was observed in~\cite{Caine}. In each case, the result is independent of $b$, but is very dif\/ferent depending on whether the domain is ${\mathbb T}_{\mathbb C}$, ${\mathbb C}$, or ${\mathbb C} P^1$.
This is because Poisson cohomology is af\/fected both by the topology of~$M$ and the degeneracy of the Poisson tensor. Rather surprisingly, in each of these cases the Poisson cohomology was found to be f\/inite dimensional and, perhaps more surprisingly, it was possible to f\/ind algebraic representatives for the Poisson cohomology classes.

As an approximation to the smooth cohomology problem, we consider the algebraic coho\-mology of $\pi_B$ on ${\mathbb C}^n$ for more general $n$. With the algebraic cohomology in hand, one might hope to compute the smooth cohomology by taking some sort of limit. Then for smooth toric varieties, which have a distinguished f\/inite atlas of af\/f\/ine charts in which any toric Poisson structure of type $(1,1)$ has the form~\eqref{pi_B_in_coordinates}, one could hope to assemble the global smooth cohomology from the local results using a Mayer--Vietoris argument.

The algebraic Poisson cohomology of $({\mathbb C}^n,\pi_B)$ is the cohomology of the dif\/ferential graded algebra $\mathcal V^{\mathbb R}={\mathbb R}[x_1,y_1,\dots,x_n,y_n]\otimes \bigwedge {\mathbb R}\langle \partial_{x_1},\partial_{y_1},\dots,\partial_{x_n},\partial_{y_n}\rangle$, the exterior algebra of multi-vector f\/ields on ${\mathbb C}^n$ (viewed as ${\mathbb R}^{2n}$) with polynomial coef\/f\/icients. The dif\/ferential $\sigma\colon \mathcal V^{\mathbb R}\to\mathcal V^{\mathbb R}$ is given by $\sigma(Y)=[Y,\pi_B]$, where $[\cdot,\cdot]$ is the Schouten bracket (\cite{Vaisman}). Since $\pi_B$ is homogeneous quadratic, $\sigma$ on the space of smooth multi-vector f\/ields leaves invariant the subspace of multi-vector f\/ields with polynomial coef\/f\/icients. It will be convenient to complexify the problem so that we can work with complex coordinates $z_k$ and $\bar{z}_k$ instead of the real coordinates $x_k$ and $y_k$. Let
\begin{gather*}
\mathcal V=\mathcal V^{\mathbb R}\otimes_{\mathbb R}{\mathbb C} \simeq {\mathbb C}[x_1,y_1,\dots,x_n,y_n]\otimes \bigwedge {\mathbb C}\langle \partial_{x_1},\partial_{y_1},\dots,\partial_{x_n},\partial_{y_n}\rangle.
\end{gather*}
Then the assignments $z_k\mapsto x_k+iy_k$, $\bar{z}_k\mapsto x_k-iy_k$, $\partial_{z_k}\mapsto \frac{1}{2}(\partial_{x_k}-i\partial_{y_k})$, and $\partial_{\bar{z}_k}\mapsto \frac{1}{2}(\partial_{x_k}+i\partial_{y_k})$ determine an isomorphism
\begin{gather*}
\mathcal V\simeq {\mathbb C}[z_1,\dots,z_n,\bar{z}_1,\dots,\bar{z}_n]\otimes \bigwedge {\mathbb C}\langle \partial_{z_1},\dots,\partial_{z_n},\partial_{\bar{z}_1},\dots,\partial_{\bar{z}_n}\rangle,
\end{gather*}
which carries the real coordinate expression for $\pi_B$ to the complex coordinate expression for~$\pi_B$. But then, at this algebraic level, computing the real algebraic cohomology of $({\mathbb C}^n,\pi_B)$ with complex coef\/f\/icients is the same computing the complex algebraic cohomology of $({\mathbb C}^{2n},\pi_B)$ where
\begin{gather}\label{general_pi_z_w_with_factor}
\pi_B=-2i\sum_{p,q=1}^n B_{pq} z_pw_q\partial_{z_p}\wedge\partial_{w_q}
\end{gather}
in terms of the complex linear coordinates $(z_1,\dots,z_n,w_1,\dots,w_n)$. Indeed, the assignments $w_k\mapsto \bar{z}_k$ determine an isomorphism of DGAs from the complex algebraic complex for $({\mathbb C}^{2n},\pi_B)$ to the real algebraic complex for $({\mathbb C}^n,\pi_B)$ but with complex coef\/f\/icients.

Since the Poisson dif\/ferential is linear in its dependence on the Poisson structure, modifying the Poisson structure by an overall non-zero scalar factor does not af\/fect the coho\-mo\-logy. So the factor of $-2i$ in \eqref{general_pi_z_w_with_factor} may be ignored. Thus, we consider the general problem of computing the cohomology of the DGA whose algebra is the exterior algebra $R\otimes\bigwedge V$ over $R={\mathbb C}[z_1,\dots,z_n,w_1,\dots,w_n]$, where $V={\mathbb C}\langle\partial_{z_1},\dots,\partial_{z_n},\partial_{w_1},\dots,\partial_{w_n}\rangle$, equipped with the dif\/ferential $\sigma$ over ${\mathbb C}$ def\/ined by $\sigma(Y)=[Y,\pi_B]$ for each $Y\in R\otimes\bigwedge V$ and
\begin{gather}\label{general_pi_z_w}
\pi_B=\sum_{p,q=1}^n B_{pq} z_pw_q\partial_{z_p}\wedge\partial_{w_q}
\end{gather}
for some Hermitian matrix $[B_{pq}]$. We will be concerned primarily with examples where $[B_{pq}]$ is invertible, so that $\pi$ is generically non-degenerate and the cohomology has a chance at being f\/inite dimensional.

\subsection[General results for $H^{0}$ and $H^{1}$]{General results for $\boldsymbol{H^{0}}$ and $\boldsymbol{H^{1}}$}

The f\/irst general result on the cohomology concerns $H^0$, which we recall is canonically isomorphic to the subspace elements of $\ker\sigma$ of wedge degree $0$.

\begin{Proposition}\label{H0_prop}
If $[B_{pq}]$ in \eqref{general_pi_z_w} is invertible, so that $\pi_B$ is generically non-degenerate, then $H^0({\mathbb C}^{2n},\pi_B)={\mathbb C}\langle 1\rangle$.
\end{Proposition}
\begin{proof}
If $[B_{pq}]$ is invertible, then $\pi_B$ is non-degenerate on the open dense set ${\mathbb T}_{\mathbb C}^{2n}\subset{\mathbb C}^{2n}$ by a~computation similar to that in the proof of Proposition~\ref{generically_nondeg_prop}. The complex algebraic functions~$f$ on~${\mathbb C}^{2n}$ with $[f,\pi_B]=0$ are those for which~$\pi_B^\#(\partial f)=0$ on~${\mathbb C}^{2n}$, where $\partial$ is the Dolbeault operator. Such functions satisfy $\pi_B^\#(\partial f)=0$ on ${\mathbb T}_{\mathbb C}^{2n}$ in particular. Since ${\mathbb T}_{\mathbb C}^{2n}$ is connected and $\pi_B$ is non-degenerate there, we see that this requires $f$ to be constant. But an algebraic function which is constant on ${\mathbb T}_{\mathbb C}^{2n}$ is constant on~${\mathbb C}^{2n}$. Hence $H^0({\mathbb C}^{2n},\pi_B)={\mathbb C}\langle 1\rangle$ for the algebraic cohomology.
\end{proof}

This argument certainly applies to holomorphic functions, so the same result is valid in the holomorphic category. By setting $w_k=\bar{z}_k$, re-inserting the factor of $-2i$ in $\pi_B$, and considering not the Dolbeault operator $\partial$ but the full de Rham dif\/ferential $d=\partial+\bar{\partial}$, the above argument applies to the smooth real problem as well. Furthermore, it applies to any complex toric mani\-fold~$M$ equipped with a generically non-degenerate real toric Poisson structure $\pi$ of type $(1,1)$. Since $\pi$ is non-degenerate on the open dense $T_{\mathbb C}$-orbit $\mathcal O$, which is dif\/feomorphic to $T_{\mathbb C}$, only the constant functions on $M$ can be constant on the symplectic leaves of $(M,\pi)$ by continuity.

The classes in $H^1(M,\pi)$ are inf\/initesimal outer automorphisms of $(M,\pi)$, i.e., classes of vector f\/ields which preserve $\pi$ ($[X,\pi]=0$) modulo those which do so because they are Hamiltonian ($X=\pi^\#(df)$). The standard basis for ${\mathbb C}^{2n}$, thought of as the Lie algebra of ${\mathbb T}_{\mathbb C}^{2n}$, determines the holomorphic vector f\/ields $z_1\partial_{z_1},\dots,z_n\partial_{z_n},w_1\partial_{w_1},\dots,w_n\partial_{w_n}$ on ${\mathbb C}^{2n}$ through the inf\/initesimal action of $\mathbb T_{\mathbb C}^n$. Such vector f\/ields preserve $\pi_B$ in \eqref{general_pi_z_w} because $\pi_B$ is a linear combination of wedge products of these f\/ields. The f\/ields are locally, but not globally, Hamiltonian because their primitives involve $\log(z_k)$ or $\log(w_k)$ which are not even single valued on ${\mathbb T}_{\mathbb C}^{2n}$ let alone extend to $z_k=0$ or $w_k=0$. In the algebraic category, a polynomial Hamiltonian then does not exist. Rather surprisingly, every other polynomial vector f\/ield in the kernel has a polynomial Hamiltonian.

\begin{Theorem}\label{BeritsH1theorem}
If $[B_{pq}]$ in \eqref{general_pi_z_w} is symmetric and invertible, then
\begin{gather*}
H^1\big({\mathbb C}^{2n},\pi_B\big)={\mathbb C}\langle z_1\partial_{z_1},\dots,z_n\partial_{z_n},w_1\partial_{w_1},\dots,w_n\partial_{w_n}\rangle
\end{gather*}
for the algebraic cohomology.
\end{Theorem}
\begin{proof}
We start by computing the action of $\sigma$ on the natural ${\mathbb C}$-basis elements of $R\otimes \bigwedge V$, i.e., those elements of the form $z^\alpha w^\beta\partial_{z_j}$ or $z^\alpha w^\beta \partial_{w_k}$. A straightforward computation with the Schouten bracket shows that
\begin{gather}\label{sigma_on_basis_1}
\sigma\big(z^\alpha w^\beta \partial_{z_k}\big)=z^\alpha w^\beta \partial_{z_k}\wedge\left(\sum_{p=1}^n (B_p\cdot\beta) z_p\partial_{z_p}-\sum_{q=1}^n (B_q\cdot (\alpha-e_k))w_q\partial_{w_q}\right)
\end{gather}
and likewise
\begin{gather}\label{sigma_on_basis_2}
\sigma\big(z^\alpha w^\beta \partial_{w_k}\big)=z^\alpha w^\beta \partial_{w_k}\wedge\left(\sum_{p=1}^n (B_p\cdot(\beta-e_k)) z_p\partial_{z_p}-\sum_{q=1}^n (B_q\cdot \alpha) w_q\partial_{w_q}\right)
\end{gather}
for each $\alpha,\beta\in\mathbb N^n$ and $k=1,2,\dots,n$, where $B_p$ denotes the $p^{th}$ row or column of $B$.

Let us next observe that we can decompose a general linear combination of basis elements in a useful manner. To describe the decomposition, it will be convenient to brief\/ly suppress the splitting of variables into $z$ and $w$ and instead using a single variable $\zeta$ for both: we write $X=\sum c_{\mu}^k \zeta^\mu\partial_{\zeta_k}$, where $\mu=(\alpha, \beta)$, $\zeta_k=z_k$ if $k\leq n$ or $w_{k-n}$ if $k>n$, and $\zeta^\mu=z^\alpha w^\beta$. We partition the terms of $X$ by def\/ining
\begin{gather*}
X_\lambda=\sum_{\substack{\mu, k\\ \mu=\lambda+e_k}} c_{\mu}^k \zeta^\mu \partial_{\zeta_k},
\end{gather*}
as $\lambda$ ranges over integral vectors of length $2n$. Although it appears in a dif\/ferent form, this decomposition technique is used in the general work of \cite{Monnier}. Note that $X_\lambda\ne 0$ for only f\/initely many~$\lambda$. Furthermore, an inspection of~\eqref{sigma_on_basis_1} and~\eqref{sigma_on_basis_2} reveals that if $\lambda$ and $\lambda'$ are distinct, then~$\sigma(X_{\lambda})$ and~$\sigma(X_{\lambda'})$ have no like terms, that is, no terms with the same wedge and homogeneous degrees. This means that if $\sigma(X)=0$, then $\sigma(X_\lambda)=0$ for each~$\lambda$.

By considering one such $X_\lambda$ at a time, we may assume without loss of generality that $X=X_\lambda$. Since for each $\partial_{\zeta_k}$ and f\/ixed $\lambda$, the exponent vector~$\mu$ is uniquely determined, we can rewrite~$X$ as
\begin{gather*}
X=\sum_{k=1}^{2n} c^k_{\lambda+e_k} \zeta^{\lambda+e_k}\partial_{\zeta^k}.
\end{gather*}
In general, the coordinates of $\lambda+e_k$ must be nonnegative, so the coordinates of $\lambda$ must be greater than or equal to $-1$ with at most one coordinate equal to $-1$. When $\lambda=0$ in ${\mathbb Z}^{2n}$, then $X_\lambda$ is a linear combination of $\zeta_1\partial_{\zeta_1},\dots,\zeta_{2n} \partial_{\zeta_{2n}}$. If $\lambda_\ell=-1$, then $X$ is the single term $X=c\zeta^{\lambda+ e_\ell}\partial_{\zeta_\ell}$ for some $c\in{\mathbb C}$.

Suppose $\lambda_\ell=-1$ with $\ell\le n$ and translate back in terms of $z$ and $w$, so that $X=cz^\alpha w^\beta \partial_{z_\ell}$ and $\alpha_\ell=0$.
Applying~\eqref{sigma_on_basis_1}, we can compute~$\sigma(X)$ and see that the coef\/f\/icient of~$\partial_{z_\ell}\wedge \partial_{w_q}$ in~$\sigma(X)$ is $(B_q\cdot (\alpha-e_\ell))w_qz^{\alpha}w^\beta$ for each $q$. Since $B$ is invertible and $\alpha-e_\ell\ne 0$, then $\sigma(X)\ne 0$. Similarly, by applying \eqref{sigma_on_basis_2} we can deduce the same conclusion if $\ell>n$. Thus if $X=X_\lambda$ is in the kernel of $\sigma$, we may assume that all coordinates of $\lambda$ are nonnegative and so
\begin{gather}\label{good_X_lambda}
X=X_\lambda=\sum_{j=1}^n a^j z^{\alpha+e_j}w^\beta \partial_{z_j}+\sum_{j=1}^n b^j_{\alpha,\beta} z^\alpha w^{\beta+e_j}\partial_{w_j}
\end{gather}
for $\lambda=(\alpha,\beta)$ and some coef\/f\/icients $a^j_{\alpha,\beta}$ and $b^j_{\alpha,\beta}$.
Applying \eqref{sigma_on_basis_1} and \eqref{sigma_on_basis_2} to \eqref{good_X_lambda}, we f\/ind that
\begin{gather}\label{sigma_on_X_lambda}
\sigma(X)=X\wedge \left(\sum_{p=1}^n (B_p\cdot \beta)z_p\partial_{z_p}-\sum_{q=1}^n (B_q\cdot \alpha) w_q\partial_{w_q}\right).
\end{gather}
Another straightforward computation with the Schouten bracket shows that
\begin{gather}\label{sigma_of_monomial}
\sigma\big(z^\alpha w^\beta\big)=z^\alpha w^\beta \left(\sum_{p=1}^n (B_p\cdot \beta)z_p\partial_{z_p}-\sum_{q=1}^n (B_q\cdot \alpha) w_q\partial_{w_q}\right).
\end{gather}
Thus, if $X=X_\lambda\not=0$ satisf\/ies $\sigma(X)=0$, then $X\wedge \sigma(z^\alpha w^\beta)=0$ by~\eqref{sigma_on_X_lambda} and \eqref{sigma_of_monomial}. Hence, if $\sigma(z^\alpha w^\beta)\not=0$ then $X=c\sigma(z^\alpha w^\beta)$ for some $c\not=0$. If $(\alpha,\beta)=(0,0)$, then $\sigma(z^\alpha w^\beta)=0$ by inspection of~\eqref{sigma_of_monomial}. On the other hand, if $(\alpha,\beta)\not=(0,0)$, then $\sigma(z^\alpha w^\beta)\not=0$ because $[B_{pq}]$ is invertible. Thus, $X=c\sigma(z^\alpha w^\beta)=\sigma(cz^\alpha w^\beta)$ when $(\alpha,\beta)\not=(0,0)$. This completes the proof.
\end{proof}

\subsection[Algebraic cohomology for $n=1$]{Algebraic cohomology for $\boldsymbol{n=1}$} When $n=1$, the matrix $[B_{pq}]$ is $1\times 1$ consisting of a single non-zero complex number. Hence $\pi_B=bzw\partial_z\wedge\partial_w$. The cohomology of $({\mathbb C}^2,bzw\partial_z\wedge\partial_w)$ has been computed by a number of methods (cf.~\cite{Goto,Monnier2}). We supply another algebraic method. As mentioned before, an overall non-zero scalar multiple of a Poisson tensor does not af\/fect its cohomology, so we can more simply consider $\pi=zw\partial_z\wedge\partial_w$. By Proposition~\ref{H0_prop} and Theorem~\ref{BeritsH1theorem}, we know that
\begin{gather*}
H^0\big({\mathbb C}^2,bzw\partial_z\wedge\partial_w\big)={\mathbb C}\langle 1\rangle \qquad \text{and} \qquad H^1\big({\mathbb C}^2,bzw\partial_z\wedge\partial_w\big)={\mathbb C}\langle z\partial_z,w\partial_w\rangle
\end{gather*}
for each $b\not=0$. It remains then to determine $H^2({\mathbb C}^2,bzw\partial_z\wedge\partial_w)$ which is the space of all bi-vector f\/ields on ${\mathbb C}^2$ modulo the subspace of the image of $\sigma$ in wedge degree $2$. At this point, let us observe that the cochain complex $R\otimes \bigwedge V$ can be graded not only by the wedge degree from the second factor $\bigwedge V$ but also using the total homogeneous degree on the polynomial ring $R$. The Poisson dif\/ferential $\sigma$, which raises the wedge degree by $1$ also raises the homogeneous degree by $1$ because its action on a multi-vector f\/ield involves 1st order dif\/ferentiation of coef\/f\/icients post-multiplied by a homogeneous quadratic expression. The cohomology $H({\mathbb C}^2,\pi)$ then has a~direct sum decomposition
\begin{gather*}
H\big({\mathbb C}^2,\pi\big)=\bigoplus_{d\in\mathbb N} \bigoplus_{p=0}^2 H^p_{[d]}\big({\mathbb C}^2,\pi\big),
\end{gather*}
where
\begin{gather*}
H^p_{[d]}\big({\mathbb C}^2,\pi\big)=\frac{\ker \sigma\colon R_{[d]}\otimes \bigwedge^p V\to R_{[d+1]}\otimes \bigwedge^{p+1} V}{\mathrm{im}\,\sigma\colon R_{[d-1]}\otimes \bigwedge^{p-1} V\to R_{[d]}\otimes \bigwedge^p V}.
\end{gather*}
In particular, the bi-vector f\/ield $\partial_z\wedge\partial_w$ which generates $R_{[0]}\otimes \bigwedge^2 V$ is certainly not in the image of $\sigma$ and so $H^2_{[0]}({\mathbb C}^2,\pi)={\mathbb C}\langle \partial_z\wedge\partial_w\rangle$. The next result completes the computation of the complex algebraic cohomology. See Fig.~\ref{Nakanishi_vanishing_table}. Intriguingly, this recovers the same result found by Nakanishi in~\cite{Nakanishi},
\begin{gather*}
H\big({\mathbb C}^2,zw\partial_z\wedge\partial_w\big)={\mathbb C}\langle 1,z\partial_z,w\partial_w,\partial_z\wedge\partial_w,zw\partial_z\wedge\partial_w\rangle
\end{gather*}
except that he considered real coef\/f\/icients, with $z$ and $w$ as real variables, and used analytic techniques to classify the solutions to the partial dif\/ferential equations encoded in $\sigma(Y)=0$.

\begin{figure}[t]\centering
$
\begin{array}{c|ccc}
\dim H_{[d]}^p & 0 & 1 & 2 \\ \hline
0 & 1 & 0 & 1 \\
1 & 0 & 2 & 0 \\
2 & 0 & 0 & 1 \\
3 & 0 & 0 & 0 \\
4 & 0 & 0 & 0 \\
5 & 0 & 0 & 0 \\
6 & 0 & 0 & 0 \\
\end{array}
$
\caption{\label{Nakanishi_vanishing_table}Dimensions of $H^p_{[d]}({\mathbb C}^2,zw\partial_z\wedge\partial_w)$ up to $d=6$.}
\end{figure}

\begin{Proposition}
For $b\not=0$,
\begin{gather*}
H^2\big({\mathbb C}^2,bzw\partial_z\wedge\partial_w\big)\simeq {\mathbb C}\langle \partial_z\wedge\partial_w,zw\partial_z\wedge\partial_w\rangle
\end{gather*}
for the complex algebraic cohomology.
\end{Proposition}
\begin{proof} As we observed before, $H({\mathbb C}^2,bzw\partial_z\wedge\partial_w)=H({\mathbb C}^2,zw\partial_z\wedge\partial_w)$. We will compute $H^2_{[d]}({\mathbb C}^2,bzw\partial_z\wedge\partial_w)$ for each $d\ge 1$. Let $d\ge 1$ be given. Since $\ker\sigma\cap (R_{[d]}\otimes \bigwedge^2 V)=R_{[d]}\otimes \bigwedge^2 V$, it suf\/f\/ices to show that every basis element of the form $z^pw^q\,\partial_z\wedge\partial_w$, where $p+q=d$ with $d\not=2$, has a primitive with respect to $\sigma$. Consider a general element $X\in R_{[d-1]}\otimes \bigwedge^1 V$ of the form
\begin{gather*}
X=\sum_{\alpha,\beta}a_{\alpha\beta}z^\alpha w^\beta\partial_z+\sum_{\gamma,\delta} b_{\gamma,\delta}z^\gamma w^\delta \partial_w,
\end{gather*}
where the sums are over pairs of non-negative integers which sum to $d-1$. A straightforward calculation with the Schouten bracket shows that
\begin{gather*}
\sigma(X) =  -\Big({}-a_{0,d-1}w^d+b_{0,d-1}(d-2)zw^{d-1}    +\sum (a_{p,q-1}(p-1)+b_{p-1,q}(q-1))z^pw^q   \\
 \hphantom{\sigma(X) =}{} +  a_{d-1,0}(d-2)z^{d-1}w-b_{d-1,0}z^d\Big)\partial_z\wedge\partial_w, %\label{Nakanishi_deg_2}
\end{gather*}
where the sum is over $p,q\ge 2$ with $p+q=d$. From this formula, it follows that the basis elements
\begin{gather*}
w^d\partial_z\wedge\partial_w,\qquad z^d\partial_z\wedge\partial_w
\end{gather*}
have unique primitives with respect to $\sigma$ and the basis elements
\begin{gather*}
zw^{d-1}\partial_z\wedge\partial_w,\qquad z^{d-1}w\partial_z\wedge\partial_w
\end{gather*}
have unique primitives with respect to $\sigma$ provided $d\not= 2$. The basis elements $z^pw^q\partial_z\wedge\partial_w$ with $p,q\ge 2$ have one parameter families of primitives. The only basis vectors of $R\otimes \bigwedge^2 V$ which are not accounted for are $\partial_z\wedge\partial_w$ and $zw\partial_z\wedge\partial_w$. As we already observed, the f\/irst cannot be in the image of $\sigma$. Examining the formula above for $\sigma(X)$ when $d=2$, we see that $zw\partial_z\wedge\partial_w$ does not appear either. Therefore, $H_{[d]}^2=\mathbf 0$ when $d\ge 1$, unless $d=2$ in which case $H_{[2]}^2={\mathbb C}\langle zw\partial_z\wedge\partial_w\rangle$.
\end{proof}

\subsection[Numerical results for $H^{*}$ when $n=2$]{Numerical results for $\boldsymbol{H^{*}}$ when $\boldsymbol{n=2}$}

When $n=2$, we are considering the Poisson cohomology $H({\mathbb C}^4,\pi_B)$ in the complex algebraic category where
\begin{gather*}
\pi_B=B_{11}z_1w_1\partial_{z_1}\wedge\partial_{w_1}+B_{12}z_1w_2\partial_{z_1}\wedge\partial_{w_2} +B_{21}z_2w_1\partial_{z_2}\wedge\partial_{w_1}+B_{22}z_2w_2\partial_{z_2}\wedge\partial_{w_2}
\end{gather*}
for some matrix
\begin{gather*}
B=\begin{pmatrix}B_{11} & B_{12} \\ B_{21} & B_{22}\end{pmatrix}.
\end{gather*}
With the increase in the number of variables and the number of terms in $\pi_B$, comes an increase in complexity of the cohomological computations, even for the algebraic problem. Still, $\sigma$ in this context raises wedge degree by $1$ and total homogeneous degree by $1$, so the cochain complex $R\otimes \bigwedge V$ decomposes as a~direct sum of cochain complexes $(\bigoplus_{p=0}^{4}R_{[d+p]}\otimes \bigwedge^p V,\sigma)$ over $d$ from~$-2n$ to~$\infty$, where we set $R_{[\ell]}\otimes \bigwedge^p V=\mathbf 0$ if $\ell <0$. While $R\otimes \bigwedge^p V$ is inf\/inite dimensional, each $R_{[d]}\otimes \bigwedge^p V$ is f\/inite dimensional with
\begin{gather*}
\dim_{\mathbb C} R_{[d]}\otimes \textstyle{\bigwedge^p V}= {d+2n-1 \choose 2n-1} {2n\choose p}, \qquad \text{where} \quad n=2.
\end{gather*}
Thus, $\sigma\colon R_{[d]}\otimes \bigwedge^p V\to R_{[d+1]}\otimes \bigwedge^{p+1} V$ has a matrix representation, and linear algebra can be used to determine each $H^p_{[d]}({\mathbb C}^4,\pi_B)$. For brevity, let's write $\sigma^p_{[d]}$ for $\sigma\colon R_{[d]}\otimes \bigwedge^p V\to R_{[d+1]}\otimes \bigwedge^{p+1} V$.

\begin{figure}[t]\centering
$B=\begin{pmatrix}1 & 0 \\ 0 & 1\end{pmatrix}\hspace{30mm} B=\begin{pmatrix}2 & 1 \\ 1 & 2\end{pmatrix}$\vspace{1mm}\\
$\begin{array}{c|ccccc}
\dim H_{[d]}^p & 0 & 1 & 2 & 3 & 4 \\ \hline
0 & 1 & 0 & 2 & 0 & 1 \\
1 & 0 & 4 & 0 & 4 & 0 \\
2 & 0 & 0 & 6 & 0 & 2 \\
3 & 0 & 0 & 0 & 4 & 0 \\
4 & 0 & 0 & 0 & 0 & 1 \\
5 & 0 & 0 & 0 & 0 & 0 \\
6 & 0 & 0 & 0 & 0 & 0 \\
7 & 0 & 0 & 0 & 0 & 0 \\
8 & 0 & 0 & 0 & 0 & 0
\end{array}\ \ \ \
\begin{array}{c|ccccc}
\dim H_{[d]}^p & 0 & 1 & 2 & 3 & 4 \\ \hline
0 & 1 & 0 & 0 & 0 & 1 \\
1 & 0 & 4 & 0 & 0 & 0 \\
2 & 0 & 0 & 6 & 0 & 0 \\
3 & 0 & 0 & 0 & 4 & 0 \\
4 & 0 & 0 & 2 & 0 & 1 \\
5 & 0 & 0 & 0 & 4 & 0 \\
6 & 0 & 0 & 0 & 0 & 2 \\
7 & 0 & 0 & 0 & 0 & 0 \\
8 & 0 & 0 & 0 & 0 & 0
\end{array}$
\caption{\label{table1_results}Cohomological dimensions up to $d=8$ for the two cases shown.}
\end{figure}

We were able to use Mathematica~10 to perform these computations for small $d$ up to $d=8$ with fairly fast computing times on a 2.7~GHz processor with 8~GB of RAM. The dimension of~$R_{[d]}\otimes \bigwedge^p V$ grows fairly quickly with $d$ and consequently the matrices representing each $\sigma^p_{[d]}$ get large fairly quickly. For example, $\sigma^2_{[8]}$ is represented by an $880\times 990$ matrix.

The computations involved a mix of symbolic computation and linear algebra, but the methods are standard. If we write $z^pw^q$ for the monomial $z_1^{p_1}z_2^{p_2}w_1^{q_1}w_2^{q_2}$ determined by $(p,q)\in \mathbb N^2\times \mathbb N^2$ and likewise write $\partial_{z}^\gamma\partial_{w}^\delta$ for $\partial_{z_1}^{\gamma_1}\wedge\partial_{z_2}^{\gamma_2}\wedge\partial_{w_1}^{\delta_1}\wedge\partial_{w_2}^{\delta_2}$ in $\bigwedge V$ determined by $(\gamma,\delta)\in \{0,1\}^2\times\{0,1\}^2$, then we obtain a ${\mathbb C}$-basis $\{z^p w^q \partial_{z}^\gamma \partial_{w}^\delta\}$ for $R_{[d]}\otimes \bigwedge^p$ indexed by all $(p,q)$ with $p_1+p_2+q_1+q_2=d$ and $(\gamma,\delta)$ with $\gamma_1+\gamma_2+\delta_1+\delta_2=p$. Determining the matrix representing $\sigma^p_{[d]}$ in terms of these bases involved symbolic computation, implementing calculations with the Schouten bracket. Once the matrix representations were found, computing $H^p_{[d]}$ reduced to the following linear algebra.
\begin{enumerate}\itemsep=0pt
\item Using Gauss--Jordan elimination, we can determine a basis for the column space of the matrix representing $\sigma^{p-1}_{[d-1]}$ and determine the corresponding symbolic basis for $\mathrm{im}\,\sigma^{p-1}_{[d-1]}$.
\item Next we f\/ind a basis for the null space of the matrix representing $\sigma^p_{[d]}$ and determine the corresponding symbolic basis for $\ker\sigma_{[d]}^p$.
\item If we form a matrix whose rows are the basis vectors of the column space for $\sigma_{[d-1]}^{p-1}$ (written as rows) followed by the basis vectors for the null space of the matrix for $\sigma^p_{[d]}$ (written as rows) and then perform Gauss--Jordan elimination on this matrix we obtain a matrix in reduced row echelon form.
\item If the basis for the column space of $\sigma_{[d-1]}^{p-1}$ had $k$ elements, then the upper left $k\times k$ block of this reduced matrix will be the identity. The remaining non-zero rows, rewritten as columns, form a basis for a subspace of the null space of the matrix for $\sigma^p_{[d]}$ which is complementary to column space of the matrix for $\sigma^{p-1}_{[d-1]}$. The corresponding symbolic vectors are representatives for a basis of $H^p_{[d]}$.
\end{enumerate}

\begin{figure}[t]\centering
$B=\begin{pmatrix}3 & -1 \\ -1 & 2\end{pmatrix}\hspace{20mm} B=\begin{pmatrix} 6 & -2 \\ -2 & 2 \end{pmatrix} \hspace{20mm} B=\begin{pmatrix}11 & -3 \\ -3 & 2 \end{pmatrix}$
\vspace{1mm}\\
$\begin{array}{c|ccccc}
\dim H_{[d]}^p & 0 & 1 & 2 & 3 & 4 \\ \hline
0 & 1 & 0 & 0 & 0 & 1 \\
1 & 0 & 4 & 0 & 0 & 0 \\
2 & 0 & 0 & 6 & 0 & 0 \\
3 & 0 & 0 & 0 & 4 & 0 \\
4 & 0 & 0 & 0 & 0 & 1 \\
5 & 0 & 0 & 0 & 0 & 0 \\
6 & 0 & 0 & 0 & 0 & 0 \\
7 & 0 & 0 & 0 & 0 & 0 \\
8 & 0 & 0 & 0 & 0 & 0
\end{array}\ \ \ \
\begin{array}{c|ccccc}
\dim H_{[d]}^p & 0 & 1 & 2 & 3 & 4 \\ \hline
0 & 1 & 0 & 0 & 2 & 1 \\
1 & 0 & 4 & 0 & 2 & 2 \\
2 & 0 & 0 & 6 & 2 & 2 \\
3 & 0 & 0 & 0 & 6 & 2 \\
4 & 0 & 0 & 0 & 2 & 3 \\
5 & 0 & 0 & 0 & 2 & 2 \\
6 & 0 & 0 & 0 & 2 & 2 \\
7 & 0 & 0 & 0 & 2 & 2 \\
8 & 0 & 0 & 0 & 2 & 2
\end{array}\ \ \ \
\begin{array}{c|ccccc}
\dim H_{[d]}^p & 0 & 1 & 2 & 3 & 4 \\ \hline
0 & 1 & 0 & 0 & 0 & 1 \\
1 & 0 & 4 & 0 & 0 & 0 \\
2 & 0 & 0 & 6 & 0 & 0 \\
3 & 0 & 0 & 0 & 4 & 0 \\
4 & 0 & 0 & 0 & 0 & 1 \\
5 & 0 & 0 & 0 & 0 & 0 \\
6 & 0 & 0 & 0 & 0 & 0 \\
7 & 0 & 0 & 0 & 0 & 0 \\
8 & 0 & 0 & 0 & 0 & 0
\end{array}$
\caption{Cohomological dimensions up to $d=8$ for the three cases shown.}\label{table2_results}
\end{figure}

We considered quadratic forms $B$ whose matrices $[B_{ij}]$ had one of f\/ive forms
\begin{gather*}
\begin{pmatrix}1 & 0 \\ 0 & 1\end{pmatrix},\qquad \begin{pmatrix}2 & 1 \\ 1 & 2\end{pmatrix},\qquad \begin{pmatrix}2+m^2 & -m \\ -m & 2\end{pmatrix}\qquad \text{for} \quad m=1,2,3.
\end{gather*}
These give local forms for the toric Poisson structures constructed in~\cite{Caine} on the smooth toric varieties ${\mathbb C} P^1\times {\mathbb C} P^1$, ${\mathbb C} P^2$, and the Hirzebruch surfaces~$X_m$ for $m=1,2,3$, respectively. Figs.~\ref{table1_results} and~\ref{table2_results} show the cohomological dimensions computed for each of these cases for $0\le p\le 4$ and $0\le d\le 8$. In each of these cases, the elements $z^\gamma w^\delta \partial_{z}^\gamma \partial_{w}^\delta$ in $R\otimes \bigwedge V$ represent classes in the basis for~$H$. Furthermore, the diagonal $\bigoplus_{k=0}^4 H_{[k]}^k=\bigwedge {\mathbb C}\langle z_1\partial_{z_1},z_2\partial_{z_2},w_1\partial_{w_1},w_2\partial_{w_2}\rangle$ is generated by the basis vector f\/ields generating the ${\mathbb T}_{\mathbb C}^{4}$ action on ${\mathbb C}^4$. Let us refer to the basis vectors of $R\otimes \bigwedge V$ of this type as those of Type~I. The additional classes in each basis were, remarkably, also represented by basis vectors of $\mathcal V$ which lie in $\ker\sigma$ but are of one of two other types. The generators of Type II were $z^\alpha w^\beta \partial_{z^\gamma w^\delta}$ in $\ker \sigma$ with the property that $\alpha\cdot\gamma=0$ and $\beta\cdot\delta=0$. Those of the remaining Type~III were non-zero wedge products of generators of Type~I and Type~II. Thus, it appears that the cohomology is generated as a module over the diagonal $\bigoplus_{k=0}^4 H_{[k]}^k=\bigwedge {\mathbb C}\langle z_1\partial_{z_1},z_2\partial_{z_2}, w_1\partial_{w_1}, w_2\partial_{w_2}\rangle$ by $1$ and the basis vectors of $R\otimes \bigwedge V$ which lie in the kernel of $\sigma$ and are of Type II. In a forthcoming paper, we use representation theory to show that such a presentation of the cohomology as a module over $\bigwedge {\mathfrak t}_{\mathbb C}$ is possible for any~$n$ and~$B$.

\begin{Example}
For $B=\begin{pmatrix}1 & 0 \\ 0 & 1\end{pmatrix}$, the cohomology up to homogeneous degree $8$ generated as a~module over the diagonal by $1$, $\partial_{z_1}\wedge\partial_{w_1}$, $\partial_{z_2}\wedge\partial_{w_2}$, $\partial_{z_1}\wedge\partial_{z_2}\wedge\partial_{w_1}\wedge\partial_{w_2}$.
\end{Example}
\begin{Example}
For $B=\begin{pmatrix}2 & 1 \\ 1 & 2\end{pmatrix}$, the cohomology up to homogeneous degree $8$ is generated as a module over the diagonal by $1$, $z_1^2w_2^2\partial_{z_2}\wedge\partial_{w_1}$, $z_2^2w_1^2\partial_{z_1}\wedge\partial_{w_2}$, and $\partial_{z_1}\wedge\partial_{z_2}\wedge\partial_{w_1}\wedge\partial_{w_2}$.
\end{Example}
\begin{Example}
For $B=\begin{pmatrix}3 & -1 \\ -1 & 2\end{pmatrix}$, the cohomology up to homogeneous degree $8$ is generated as a module over the diagonal by $1$ and $\partial_{z_1}\wedge\partial_{z_2}\wedge\partial_{w_1}\wedge\partial_{w_2}$.
\end{Example}
\begin{Example}
For $B=\begin{pmatrix}6 & -2 \\ -2 & 2\end{pmatrix}$, the cohomology up to homogeneous degree $8$ is generated as a module over the diagonal by $1$ and $\partial_{z_1}\wedge\partial_{z_2}\wedge\partial_{w_1}\wedge\partial_{w_2}$ and $z_2^d\partial_{z_1}\wedge\partial_{w_1}\wedge\partial_{w_2},w_2^d\partial_{z_1}\wedge\partial_{z_2}\wedge\partial_{w_1}$ for each $0\le d\le 8$. In fact, we can show that these latter expressions are generators of the full cohomology for every degree $d$. Thus, the algebraic cohomology of this Poisson structure on ${\mathbb C}^4$ is inf\/inite dimensional despite the fact that $B$ is integral and symmetric positive def\/inite.
\end{Example}
\begin{Example}
For $B=\begin{pmatrix}11 & -3 \\ -3 & 2\end{pmatrix}$, the cohomology up to homogeneous degree $8$ is generated as a module over the diagonal by $1$ and $\partial_{z_1}\wedge\partial_{z_2}\wedge\partial_{w_1}\wedge\partial_{w_2}$.
\end{Example}

\subsection*{Acknowledgements}

Portions of this work were completed independently by the two authors during independent sabbatical leaves from California State Polytechnic University Pomona and, separately, while supported by the Provost's Teacher-Scholar Program. We appreciate this support and the suggestions from the referees which improved the paper.

\pdfbookmark[1]{References}{ref}
\LastPageEnding


\begin{thebibliography}{99}
\footnotesize\itemsep=0pt

\bibitem{Caine}
Caine A., Toric {P}oisson structures, \textit{Mosc. Math.~J.} \textbf{11}
 (2011), 205--229, \href{http://arxiv.org/abs/0910.0229}{arXiv:0910.0229}.

\bibitem{Goto}
Goto R., Unobstructed deformations of generalized complex structures induced by
 {$C^\infty$} logarithmic symplectic structures and logarithmic {P}oisson
 structures, in Geometry and Topology of Manifolds, \href{https://doi.org/10.1007/978-4-431-56021-0_9}{\textit{Springer Proc.
 Math. Stat.}}, Vol.~154, Springer, Tokyo, 2016, 159--183, \href{http://arxiv.org/abs/1501.03398}{arXiv:1501.03398}.

\bibitem{LiuXu}
Liu Z.J., Xu P., On quadratic {P}oisson structures, \href{https://doi.org/10.1007/BF00420516}{\textit{Lett. Math. Phys.}}
 \textbf{26} (1992), 33--42.

\bibitem{Lu}
Lu J.-H., Momentum mappings and reduction of {P}oisson actions, in Symplectic
 Geometry, Groupoids, and Integrable Systems ({B}erkeley, {CA}, 1989),
 \href{https://doi.org/10.1007/978-1-4613-9719-9_15}{\textit{Math. Sci. Res. Inst. Publ.}}, Vol.~20, Springer, New York, 1991,
 209--226.

\bibitem{Monnier}
Monnier P., Formal {P}oisson cohomology of quadratic {P}oisson structures,
 \href{https://doi.org/10.1023/A:1015513632414}{\textit{Lett. Math. Phys.}} \textbf{59} (2002), 253--267.

\bibitem{Monnier2}
Monnier P., Poisson cohomology in dimension two, \href{https://doi.org/10.1007/BF02773163}{\textit{Israel~J. Math.}}
 \textbf{129} (2002), 189--207.

\bibitem{Nakanishi}
Nakanishi N., Poisson cohomology of plane quadratic {P}oisson structures,
 \href{https://doi.org/10.2977/prims/1195145534}{\textit{Publ. Res. Inst. Math. Sci.}} \textbf{33} (1997), 73--89.

\bibitem{Vaisman}
Vaisman I., Lectures on the geometry of {P}oisson manifolds, \href{https://doi.org/10.1007/978-3-0348-8495-2}{\textit{Progress
 in Mathematics}}, Vol.~118, Birkh\"auser Verlag, Basel, 1994.

\end{thebibliography}
\end{document}